\documentclass[a4paper, 11pt]{article}
\usepackage{latexsym,bm}
\usepackage{mathrsfs,amsmath,amssymb,amsthm}
\usepackage[dvipdfm]{graphicx}
\usepackage[latin1]{inputenc}

\usepackage{latexsym}

\usepackage{amssymb}
\usepackage{amsfonts}
\usepackage{amsmath}
\usepackage{amsthm}
\usepackage{textcomp}
\usepackage[usenames]{color}

\usepackage{authblk}

\usepackage{comment}

\newtheorem{theorem}{Theorem}[section]
\newtheorem{lemma}[theorem]{Lemma}

\newtheorem{conj}{Conjecture}[section]

\newtheorem{corollary}[theorem]{Corollary}

\theoremstyle{definition}
\newtheorem{definition}[theorem]{Definition}
\newtheorem{ex}[theorem]{Example}
\newtheorem{remark}[theorem]{Remark}

\def\R{{\mathbb R}}    
\def\N{{\mathbb N}}      
\def\C{{\mathbb C}}      

\def\P{{\mathbb P}}
 \def\oli{\overline}
\def\Hff{\mathcal{H}}

\def\O{\Omega}

\title{Non compact boundaries of complex analytic varieties in Hilbert spaces}

\author{Samuele \textsc{Mongodi}\thanks{Email: \texttt{s.mongodi@sns.it}}}
\affil{Scuola Normale Superiore\authorcr
 Piazza dei Cavalieri 7, I-56123 Pisa, Italy}

\author{Alberto \textsc{Saracco}
\thanks{Email: \texttt{alberto.saracco@unipr.it}}}
\affil{Dipartimento di Matematica e Informatica, Universit\`a di Parma\authorcr
 Parco Area delle Scienze 53/A, I-43124 Parma, Italy}

\begin{document}


\maketitle

\abstract{We treat the boundary problem for complex varieties with isolated singularities, of complex dimension greater than or equal to $3$, non necessarily compact, which are contained in strongly convex, open subsets of a complex Hilbert space $H$. We deal with the problem by cutting with a family of complex hyperplanes and applying the already known result for the compact case.\\
\\
\noindent{\bf Classification: } Primary 32V15; Secondary 58B12}

\section {Introduction}
Let $M$ be a smooth and oriented $(2m+1)$-dimensional real submanifold of some
complex manifold $X$. A natural question
arises, whether $M$ is the boundary of an ($m+1$)-dimensional complex analytic
subvariety of $X$. This problem, the so-called \textit{boundary
problem}, has been widely treated over the past fifty-five years when
$M$ is compact and $X$ is $\C^n$ or $\C\P^n$. For a review of the boundary problem see~\cite{S}, chapter 6.

The case when $M$ is a compact, connected curve in $X=\C^n$
($m=0$), has been first solved by Wermer~\cite{W} in 1958. In 1966, Stolzenberg \cite{St} proved the same result when $M$ is a union of smooth curves. Later
on, in 1975, Harvey and Lawson in~\cite{HL} and~\cite{HL3} solved
the boundary problem in $\C^n$ and then in
$\C\P^n\setminus\C\P^r$, in terms of holomorphic chains, for any $m$. The
boundary problem in $\C\P^n$ was studied by Dolbeault and Henkin,
in~\cite{DH} for $m=0$ and in~\cite{DH2} for any $m$. Moreover, in
these two papers the boundary problem is dealt with also for
closed submanifolds (with negligible singularities) contained in
$q$-concave (i.e.\ union of $\C\P^q$'s) open subsets of $\C\P^n$.
This allows $M$ to be non compact. The results in~\cite{DH}
and~\cite{DH2} were extended by Dinh in~\cite{D99}.\vspace{0,3cm}

The main theorem proved by Harvey and Lawson in~\cite{HL} is that
if $M\subset\C^n$ is compact and maximally complex then $M$ is the
boundary of a unique holomorphic chain of finite
mass~\cite[Theorem 8.1]{HL}. Moreover, if $M$ is contained in the
boundary $b\Omega$ of a strictly pseudoconvex domain $\Omega$, then
$M$ is the boundary of a complex analytic subvariety of $\Omega$,
with isolated singularities~\cite{HL2} (see also~\cite{G}).

In~\cite{DS1} Della Sala and the second named author generalized this last theorem to a non
compact, connected, closed and maximally complex submanifold $M$ (of real dimension at least $3$, i.e.\ $m\geq1$) of
the connected boundary $b\Omega$ of an unbounded weakly
pseudoconvex domain $\Omega\subset \C^n$. The extension result is obtained via a method of \lq\lq cut-extend-and-paste\rq\rq.

In~\cite{M} the first named author established the Harvey-Lawson theorem for maximally complex manifolds of real dimension at least $3$ ($m\geq1$) contained in a complex Hilbert space, under the addition of a technical hypothesis.\vspace{0,3cm}

The aim of this paper is to combine the techniques of these last papers, in order to generalize the extension result to a non
necessarily bounded, connected, closed and maximally complex submanifold $M$ ($\dim_\R M\geq 5$, i.e.\ $m\geq2$) of
the connected boundary $b\Omega$ of a strongly convex unbounded domain $\Omega$ of a complex Hilbert space $H$. The precise definitions will be given in the following section. The main theorem we establish is the following:

\begin{theorem}\label{main} Let $H$ be a complex Hilbert space, and $M\subset H$ such that
\begin{itemize}
\item[(i)]  $M$ is a smooth maximally complex manifold of real dimension $2m+1\geq5$ (complex dimension $m\geq2$);
\item[(ii)] $M\subset b\Omega\subset H$, where $\Omega$ is a strongly convex domain;
\item[(iii)] there exists an orthogonal decomposition
$H = \mathbb C^{m+1}\times H'$ such that the orthogonal projection $p : H \to \mathbb C^{m+1}$, when restricted to $M$, is a closed immersion with transverse self-intersections;
\item[(iv)] $M$ is quasi-locally compact.
\end{itemize}
Then there exists a unique analytic chain of finite dimension
$T$ in $\Omega$ with isolated singularities, such that the boundary of $T$ is $M$.\end{theorem}

As already mentioned, the case when $\Omega$ is bounded was treated in \cite{M}; therefore we will always suppose $\Omega$ to be unbounded.

The strategy behind the proof of Theorem~\ref{main} is similar to that used in~\cite{DS1} and it is actually a simplification of that one. First we get a local and semi-global extension (see section 3), through an Lewy-type extension theorem for Hilbert-valued $CR$-functions. Then we cut $\Omega$ with parallel complex-hyperplanes.

Hypotheses (ii) and (iv) are technicalities needed in order to assure that the slices of $M$ are compact, so that we can apply the extension result in~\cite{M} (the slices of $M$ are maximally complex, and property (iii) is inherited by the hyperplane). The high dimension of $M$ is needed in order to get the maximal complexity of slices (in an Hilbert space a moments condition makes less sense than in $\mathbb C^n$). We give a simple example (see example~\ref{parapertissima}) showing that relaxing hypothesis (ii) can lead to the slices of $\Omega$ (thus of $M$) being unbounded. On the other hand, hypothesis (iv) is unnecessary (because always satisfied) if the following topological conjecture (by Williamson and Janos, 1987~\cite{WJ}) is true.

\begin{conj}\label{conj87} A complete admissible metric $d$ for a $\sigma$-compact, locally compact
space $X$ is always a Heine-Borel metric if
$$Cl\{x\in X\ |\ d(x,x_0)<r\}\, =\, \{x\in X\ |\ d(x,x_0)\leq r\},\  \forall x_0\in X, \ \forall r > 0\,.$$
\end{conj}

Using the conjecture of Williamson and Janos, we can get rid of one annoying technical hypothesis:

\begin{theorem} Assume Conjecture~\ref{conj87} is true.

Let $H$ be a complex Hilbert space, and $M\subset H$ such that
\begin{itemize}
\item[(i)]  $M$ is a smooth maximally complex manifold of real dimension $2m+1\geq5$ (complex dimension $m\geq2$);
\item[(ii)] $M\subset b\Omega\subset H$, where $\Omega$ is a strongly convex domain;
\item[(iii)] there exists an orthogonal decomposition
$H = \mathbb C^{m+1}\times H'$ such that orthogonal the projection $p : H \to \mathbb C^{m+1}$, when restricted to $M$, is a closed immersion with transverse self-intersections.
\end{itemize}
Then there exists a unique analytic chain of finite dimension
$T$ in $\Omega$ with isolated singularities, such that the boundary of $T$ is $M$.
\end{theorem}

\section{Notations and definitions}

In the following, we denote by $H$ a complex Hilbert space and by $B(x,\rho)$ the (open) ball of center $x\in H$ and radius $\rho>0$. We introduce the following \emph{quasi-local} property (following the terminology of \cite{M1}).

\begin{definition}\label{def_strcvx1}We say that $K\subset H$ is \emph{quasi-locally compact} if, for any $x\in H$, for any $\rho>0$, the set $B(x,\rho)\cap K$ is relatively compact in $H$.\end{definition}

\begin{definition}\label{def_strcvx2}Given an open set $\Omega\subset H$ with smooth boundary, we call it \emph{strongly convex at $x\in b\Omega$} if the Hessian form of the boundary at $x$ satisfies $\mathrm{Hess}_x(\cdot, \cdot)\geq \varepsilon\|\cdot\|^2$ for some fixed $\varepsilon>0$.

We call $\Omega\subset H$ \emph{strongly convex} if it is strongly convex at all its boundary points.
\end{definition}

Let $x\in b\Omega$ be a point of strong convexity, then for every $2$-dimensional real plane $P$ containing the normal to $b\Omega$ at $x$ the set $P\cap\Omega$ is a convex set locally (around $x$) contained in a parabola. Considering the cone delimited by two tangents lines to such a parabola, symmetric with respect to its axis and close enough to $x$, we note that $P\cap\Omega$ lies inside such a cone, by convexity; the angle of such a cone depends only on the $\varepsilon$ in the definition of strong convexity.

This holding for for every plane $P$ and the angle of the cone not depending on $P$, we have that $\Omega$ is contained in a cone of $H$ with fixed angle. Given a real hyperplane $L$ intersecting the cone in a bounded set, all its translations along the axis of the cone will intersect the cone in bounded sets, therefore if $L$ intersects $\Omega$ just in $x$ and is tangent to $b\Omega$, then all its translations alond the axis of the cone will intersect $\Omega$ in bounded sets.

Moreover, if $\Omega$ is unbounded, $\nu$ is the unit vector pointing in the direction of the axis of the cone (with the correct orientation) and $C_0\in\R$ is such that $(L+C_0\nu)\cap \Omega\neq \emptyset$, then the intersection $(L+C\nu)\cap\Omega$ is non-empty for $C\geq C_0$. Indeed, if $(L+C_1\nu)\cap \Omega=\emptyset$ for $C_1>C_0$, then $(L+C\nu)\cap\Omega=\emptyset$ for every $C>C_1$ by convexity.

Therefore $\Omega$ is contained in the intersection between the cone constructed above and $\{L<C_1\}$; such an intersection is bounded, which is an absurd.

\begin{definition}We recall one of the equivalent definitions of finite-dimensional analytic subvariety of an infinite dimensional complex space: $A\subset H$ is said to be a \emph{finite-dimensional analytic subvariety} if for any $x\in H$ there exist an open neighborhood $U$ of $x$ and a finite-dimensional complex manifold $W$ of $U$ such that $A\cap U\subset W$ and $A\cap U$ is an analytic subvariety of $W$. 
\end{definition}

See \cite{M} for some examples and a discussion of the relations between this definitions and the others that can be found in the literature.

\medskip

In this paper, $M$ will denote a smooth finite-dimensional manifold in $H$, of real dimension $2m+1$ greater than or equal to $5$ and $p$ will always be the projection whose existence is required in the third hypothesis of Theorem \ref{main}. $H_x(M)$ will be the holomorphic tangent to $M$ at $x$, i.e. $H_x(M)=T_xM\cap JT_xM$, where $J$ is the natural complex structure on $T_xH\cong H$.

$M$ will also be required to be \textit{maximally complex}, i.e.\ such that $\dim_{\mathbb C}\, H_x(M)=m$ at all points $x\in  M$, since maximal complexity is a necessary condition for being the boundary of a complex variety.

\medskip

Given a smooth real hypersurface $S$ in $H$, we denote by $\mathcal L_x(S)$
the Levi form of $S$ at the point $x$; we note that, if $S$ is the boundary of a strongly convex open set $\Omega$, then $\mathcal{L}_x(S)$ is positive definite for every $x\in S$, i.e. a strongly convex open set is strongly pseudoconvex.

\section{The local and semi-global results}
The aim of this section is to prove the local result.  Let $0$ be a point of $M\subset S$. We
have the following inclusions of tangent spaces: $$ H\
\supset\ T_0(S)\ \supset\ H_0(S)\ \supset\ H_0(M);$$ $$
\phantom{\C^n\ \supset}\ T_0(S)\ \supset\ T_0(M)\ \supset\
H_0(M).$$

\begin{lemma}\label{wk} Let $M$ be a maximally complex submanifold of a smooth real hypersurface $S\subset H$, $\dim_{\mathbb R}(M)=2m+1$, $m\geq1$. Suppose that, for any finite dimensional complex subspace $V$ of $H$, the restriction $\mathcal L_0(S\cap V)$ has all but at most $m$ eigenvalues of the same sign. Then
$$H_0(S)\not\supset T_0(M).$$
\end{lemma}

\begin{proof} We will reduce problem to the finite dimensional case proved in~\cite{DS1} (Lemma 3.1; for a different proof see also Lemma 2.3.1 in~\cite{TDS}).
Let us consider the following orthogonal decomposition of $T_0(H)$:
$$T_0(H) \ =\ T_0(S) \oplus l\,,$$
and let $L=T_0(M)\oplus l$ and $L^{\mathbb C}$ be the smallest complex subspace of $T_0(H)$ containing $L$.

Let $M_1=L\cap S$ (so $T_0(M)=T_0(M_1)$ and $H_0(M)=H_0(M_1)$) and $S_1=L^{\mathbb C}\cap S$ so $H_0(s)\cap L^{\mathbb C}=H_0(S_1)$.

$M_1\subset S_1$ satisfy the hypothesis of Lemma 3.1 in [3], so $H_0(S_1)\not\supset T_0(M_1)=T_0(M)$. Hence, since $H_0(s)\cap L^{\mathbb C}=H_0(S_1)$, $H_0(S)\not\supset T_0(M)$.
\end{proof}

The following lemma is an immediate consequence of a well-known fact. 

\begin{lemma}\label{Lewy} Let $D,\ D'$ be domains in a topological space $X$, with $D\subseteq D'$. Suppose we are given two Banach algebras (over $\mathbb K=\mathbb R,\mathbb C$) of continuous functions, $\mathcal{A}(D)$ and $\mathcal{B}(D')$ such that the restriction map
$$\mathcal{B}(D')\ni f\mapsto f\vert_D\in\mathcal{A}(D)$$
is bijective. Then
$$\parallel \hat{f}\parallel_{\mathcal B(D')}\ =\ \parallel f\parallel_{\mathcal A(D)}\,,$$
where $\hat{f}$ is the unique extension of $f$.
\end{lemma}
\begin{proof} Let $f\in \mathcal A(D)$. Let $x$ be any point in $D'$. We can define
$$\chi_x:A(D)\to\mathbb K\ \ \ \ \chi_x(f)=\hat{f}(x),$$
where $\hat{f}$ is the unique extension of $f$. $\chi_x$ is a character of the Banach algebra $\mathcal A(D)$, therefore continuous of unitary norm. Thus
$$|\hat {f}(x)|\ =\ |\chi_x(f)|\ \leq\ \parallel f\parallel_{\mathcal A(D)},\ \ \ \forall x\in D'\,.$$
Hence the thesis.
\end{proof}

The typical setting in which the previous lemma applies is when one is concerned with extension of analytic functions.

\begin{lemma}\label{L3} In the setting of Lemma \ref{wk}, assume that $S$
is the boundary of an unbounded domain $\O\subset H$, $0\in M$
and that the Levi form of $S$ has at most $m$ non-positive
eigenvalues. Then
\begin{enumerate}
\item[\emph{(i)}] there exists an open neighborhood $U$ of $0$ and an $(m+1)$-dimensional complex submanifold $W_0\subset U$
 with boundary, such that $bW_0=M\cap U$;
\item[\emph{(ii)}] $W_0\subset\O\cap U$.
\end{enumerate}
\end{lemma}
\begin{proof} 

To prove the first assertion, observe that to obtain $\mathcal L^M_0(\zeta_0,\oli \zeta_0)$
it suffices to choose a smooth local section $\zeta$ of $H_0(M)$
such that $\zeta(0) = \zeta_0$ and compute the
projection of the bracket $[\zeta,\oli\zeta](0)$ on the real part
of $T_0(M)$. By hypothesis, the intersection of the space
where $\mathcal L_0(S)$ is positive with $H_0(M)$ is non empty; take $\eta_0$ in this intersection. Then $\mathcal L_0^M(\eta_0, \oli\eta_0)\neq 0$. Suppose, by contradiction, that the bracket $[\eta,\oli\eta](0)$ lies in $H_0(M)$, i.e.\ its projection on the real part of the tangent of $M$ is zero. Then, if $\widetilde{\eta}$ is a local smooth extension of the field $\eta$ to $S$, we have $[\widetilde{\eta},\oli{\widetilde{\eta}}](0)= [\eta,\oli\eta](0)\in H_0(M)$. Since $H_0(M)\subset H_0(S)$, this would mean that the Levi form of $S$ in $0$ is zero in $\eta_0$.
Now, we project (generically) $M$ over a $\C^{m+1}$ in such a way that the
projection $\pi: H\to\mathbb C^{m+1}$ is a local embedding of $M$ near $0$: since the
restriction of $\pi$ to $M$ is a $CR$ function, and since the Levi
form of $M$ has - by the arguments stated above - at least one
positive eigenvalue, it follows that the Levi form of $\pi(M)$ has
at least one positive eigenvalue. Thus, in order to obtain $W_0$,
it is sufficient to apply the Lewy extension theorem \cite{Le} to
the $CR$ function $\pi^{-1}|_{\pi(M)}$.

In order to ensure that the extension lies in the Hilbert space $H$, we consider the orthonormal decomposition found before $H=\mathbb C^{m+1}\times H'$. Let $\mathbf{e}_j$ be a complex base of $H'$, and $\pi^{-1}_j(M)$  the $\mathbf{e}_j$ coordinate of $\pi^{-1}|_{\pi(M)}$. We can apply Lewy's theorem to extend all of the functions $\pi^{-1}_j$ to a fixed one-sided open neighbourhood $U$ of $0\in\pi(M)$; let us provisionaly denote by $p_j$ the extension of $\pi^{-1}_j$.

For any positive integer $k$, for any $k-$tuple $(i_1,\ldots, i_k)$ and for any $a\in\C^k$, the scalar function
$$f=a_1\pi^{-1}_{i_1}+\ldots+a_k\pi^{-1}_{i_k}$$
extends to $U$ by
$$F=a_1p_{i_1}+\ldots+a_kp_{i_k}\;.$$
Therefore, by Lemma~\ref{Lewy}, we know that
$$\|F\|_{U}\leq\|f\|_{\pi(M)}\leq\|a\|_{\C^k}\left\|\left(\sum_{j=1}^k|\pi^{-1}_{i_j}|^2\right)^{1/2}\right\|_{\pi(M)}.$$
Let $z_0\in U$ and take $\overline{a}_j=p_{i_j}(z_0)$ for $j=1,\ldots, k$. Then
$$|p_{i_1}(z_0)|^2+\ldots+|p_{j_k}(z_0)|^2\leq\|\overline{p_{i_1}(z_0)}p_{i_1}(z)+\ldots+\overline{p_{i_k}(z_0)}p_{i_k}(z)\|_{U}$$
$$\leq\left(|p_{i_1}(z_0)|^2+\ldots+|p_{j_k}(z_0)|^2\right)^{1/2}\left\|\left(\sum_{j=1}^k|\pi^{-1}_{i_j}|^2\right)^{1/2}\right\|_{\pi(M)}$$
and, letting $z_0$ vary in $U$, we have
$$\left\|\left(\sum_{j=1}^k|p_{i_j}|^2\right)^{1/2}\right\|_{U}\leq\left\|\left(\sum_{j=1}^k|\pi^{-1}_{i_j}|^2\right)^{1/2}\right\|_{\pi(M)}.$$
This implies that, if the sequence of the partial sums of
$$\sum_i {\bf e_i}\pi^{-1}_i$$
is a Cauchy sequence on $\pi(M)$ with respect to the supremum norm, then the same holds true for the sequence of partial sums of
$$\sum_i {\bf e_i} p_i$$
on $U$ with respect to the supremum norm, implying the convergence of the latter to a holomorphic map from $U$ to $H'$.

As for the second statement, we observe that the projection by
$\pi$ of the normal vector of $S$ pointing towards $\O$ lies into
the domain of $\C^{m+1}$ where the above extension $W_0$ is
defined. Indeed, the extension result in \cite{Le} gives a holomorphic function in the connected component of (a neighborhood
of $0$ in) $H \setminus \pi(M)$ for which $\mathcal
L_0(\pi(M))$ has a positive eigenvalue, when $\pi(M)$ is oriented
as the boundary of this component. This is precisely the component
towards which the projection of the normal vector of $S$ points,
when the orientations of $S$ and $M$ are chosen accordingly. This
fact, combined with Lemma \ref{wk} (which states that any
extension of $M$ must be transverse to $S$) implies that locally
$W_0\subset\O\cap U$.
\end{proof}

\begin{remark}The open neighborhood $U$ of $O$ is not uniquely determined, but the germ of complex submanifold given by $(U,W)_O$ is. This follows from the uniqueness of extension of CR functions.\end{remark}

\begin{corollary}[Semi global existence]\label{L4} Let $M$ be a maximally complex submanifold of the smooth boundary $S$ of an unbounded domain $\O\subset H$, $0\in M$, $\dim_{\mathbb R}(M)=2m+1$, $m\geq1$. Assume that
\begin{enumerate}
\item $M$ is quasi-locally compact;
\item the Levi form of $S$ has at most $m$ non-positive
eigenvalues, at each point of $M$.
\end{enumerate}
Then there exist an open tubular neighborhood $I$ of
$S=b\Omega$ in $\oli \O$ and an $(m+1)$-dimensional complex submanifold $W_0$ of
\ $\oli\O \cap I$, with boundary, such that $S\cap bW_0=M$. \end{corollary}
\begin{proof} By Lemma~\ref{L3}, for each point $x\in M$, there exist a
neighborhood $U_x$ of $x$ and a complex manifold
$W_x\subset\oli\O\cap U_x$ bounded by $M$. Since $M$ is quasi-locally compact, we can cover $M$ with
countable many such open sets $U_i$, and consider the union
$W_0=\cup_i W_i$. $W_0$ is contained in the union of the $U_i$'s,
hence we may restrict it to a tubular neighborhood $I_M$ of $M$.
It is easy to extend $I_M$ to a tubular neighborhood $I$ of $S$.

The fact that $W_{i}|_{U_{ij}}=W_{j}|_{U_{ij}}$ if $U_i\cap
U_j=U_{ij}\neq\emptyset$ immediately follows from the construction made in
Lemma~\ref{L3}, in view of the uniqueness of the holomorphic extension of
$CR$ functions. \end{proof}

\section{The global result}

In this section we will prove Theorem~\ref{main} in the case when $\Omega$ is unbounded.

Since $\Omega$ is strongly convex, by the discussion following Definitions \ref{def_strcvx1} and \ref{def_strcvx2}, we can find a real hyperplane $$I=\{\mathfrak{Re}\,\lambda\,=\,0\}$$ with $\lambda$ a complex linear functional, tangent to $b\Omega$ in $0$, such that, for every translation
$$I_a=\{\mathfrak{Re}\,\lambda\,=\,a\},\ \ \ a\in\mathbb R^+$$
of $I$, $I_a\cap\overline\Omega$ is bounded (and not empty) and the same holds for nearby hyperplanes. Denoting by
$$L_k=\{\lambda\,=\,k\},\ \ \ k\in\mathbb C,\, \mathfrak{Re}\,k\in \mathbb R^+\,,$$
also $A_k=L_k\cap M$ is bounded and, by the quasi-local compactness of $M$, the slice $A_k$ of $M$ is compact. In view of Sard's lemma, up to modifying the equation of $L_k$ by means of another complex linear functional $\mu$
$$L_k=\{\lambda+\varepsilon \mu\,=\,k\}$$
we can suppose the slice $A_k$ to be smooth and a transversal intersection, hence of the correct dimension ($2m-1$). As a notation, we'll call suitable a slicing hyperplane $L_k$ that leads to a smooth, compact, transversal intersection, as above.

Thanks to the maximal complexity of $M$, it follows that each slice $A_k$ is maximally complex too. Moreover, the technical hypothesis (iii) of Theorem~\ref{main} is inherited by the slice.

Fix a point in $\Omega$. To this correspond a suitable slicing hyperplane $L_{k_0}$ of the above form, such that nearby parallel hyperplanes are suitable too. Each slice $A_k$, $k$ in a neighborhood $U$ of $k_0$, satisfies the hypotheses of the theorem in~\cite{M}, thus is the boundary of a holomorphic chain $\tilde{A}_k$ with support in the hyperplane $L_k$, which is a smooth manifold near $b\Omega$, since there it coincides with the manifold obtained in Corollary~\ref{L4}.

Our goal is now to glue together the slices $\tilde{A}_k$: $\tilde{A}_U=\cup_{k\in U}\tilde{A}_k$ and show that $\tilde{A}_U$ is a holomorphic chain, too (without singularities near $b\Omega$, due to Corollary~\ref{L4}).
\vspace{0.3cm}

It is worth observing that a strictly convexity hypothesis does not suffice to use our slicing method, as the following example shows.

\begin{ex}\label{parapertissima} Let $\Omega\subset H$ be defined by
$$\Omega\ =\ \left\{(z_n)_{n\in\N}\in H\ |\ x_0\ >\ y_0^2+\sum_{n=1}^\infty\frac{|z_n|^2}{n}\right\}\,.$$
Then $\Omega$ is strictly convex (i.e.\ convex and its boundary does not contain lines or line segments). But it is not strongly convex at $0\in b\Omega$.

Observe that the real tangent hyperplane
$$T_0b\Omega\ =\ \left\{(z_n)_{n\in\N}\in H\ |\ x_0\ =\ 0\right\}$$
is such that all its translated in the positive $x_0$ direction intersect $\Omega$ and $b\Omega$ in unbounded sets. That is also true for complex hyperplanes of the form
$$L_k\ =\ \left\{(z_n)_{n\in\N}\in H\ |\ z_0\ =\ k\right\}, \ \ \ \mathfrak{Re}\,k>0\,.$$

Hence it is not possible to apply out slicing method to a maximally complex manifold $M\subset b\Omega$, since we have no way to assure even the boundedness of the slice.
\end{ex}

\medskip

If we show that $\widetilde{A}_U=\bigcup_{k\in U}\widetilde{A}_k$ is an analytic space, the thesis will follow. By \cite[Remark 5.4]{M}, $\widetilde{A}_U$ is a continuous family in the parameter $k$, therefore it is a rectifiable set of real dimension $2m+2$. We denote by $[\widetilde{A}_U]$ the current of integration associated to it and we define the map $\kappa:\widetilde{A}_U\to U\subset \C$ such that $x\in \widetilde{A}_{\kappa(x)}$ for every $x\in \widetilde{A}_U$; the map $\kappa$ is Lipschitz-continuous, therefore
$$\int_{\widetilde{A}_U}C(d^{\widetilde{A}_U}\kappa_x)\omega=\int_{U}\int_{\widetilde{A}_k}\omega$$
by the Coarea formula in \cite[Theorem 9.4]{AK}.

The previous formula implies that the current $[\widetilde{A}_U]$ is of bidimension $(m+1,m+1)$, which is equivalent to the fact that $\mathrm{Tan}^{(2m+2)}(\widetilde{A}_U, x)=V_x$ is a complex subspace for $\Hff^{2m+2}-$a.e. $x\in \widetilde{A}_U$; moreover, $\kappa$ is the restriction to $\widetilde{A}_U$ of a $\C$-linear map $f:H\to\C^2$, therefore $d^{\widetilde{A}_U}\kappa_x=df\vert_{V_x}$. By formula (9.2) in \cite{AK} and the properties of $\C-$linear maps, we get
$$C(d^{\widetilde{A}_U}\kappa_x)=C(df\vert_{V_x})>0\;.$$
This implies that $[\widetilde{A}_U]$ is a positive current.

The topological boundary of $\widetilde{A}_U$ is given by the union
$$\bigcup_{k\in U}A_k\cup \bigcup_{k\in bU} \widetilde{A}_k$$
and therefore is again a rectifiable set, this time of dimension $2m+1$. The boundary of the current $[\widetilde{A}_U]$ is concentrated on such a set and is therefore not contained in the bounded open set
$$\Omega_U=\bigcup_{k\in U}(L_k\cap \Omega)\;.$$

Summing up, $[\widetilde{A}_U]$ is a $(2m+2)-$rectifiable current, which is positive and closed in $\Omega_U$; therefore, by \cite[Theorem  4.5]{M}, $[\widetilde{A}_U]$ can be represented by integration on the regular part of an analytic set. We denote such a set by $V$.

\medskip

Let us consider the projection $p:H\to\C^{m+1}$, which is an immersion with self-transverse intersections when restricted to $M$, and let us suppose that, for some open set $\Omega_U$, we can find a linear functional $\nu: \pi(\Omega_U)\to\C$ such that $p(A_k)=\nu^{-1}(k)\cap p(M)$ for every $k\in U$ and such that $p\vert_{A_k}$ is again an immersion with self-transverse intersections.

We can always find such a $\nu$, up to shrinking $U$; we can also restrict $U$ further, so that every connected component $U_j$ of $\pi(\Omega_U)\setminus p(M)$ intersects $\nu^{-1}(k)$ in a non empty set for every $k\in U$. Going through the proof of Theorem 5.6 in \cite{M}, we can construct holomorphic functions
$$F^h_{j,k}:U_j\cap \nu^{-1}(k)\to H'$$
which realize $\widetilde{A}_k$ as their graph. What we proved before is that
$$U_j\ni(z,k)\mapsto F_{j,k}^h(z)=F^h_j(z,k)$$
is a holomorphic function whenever $(z,F^{h}_{j,k}(z))$ belongs to $(\widetilde{A}_k)_\textrm{reg}$; therefore, we have an analytic set $S_{j,k}\subset U_j\cap\nu^{-1}(k)$ outside which the dependence from $k$ is analytic. Let us denote by $S_j=\bigcup_k S_{j,k}$.

By an easy coarea argument, we observe that $\Hff^{2m+1}(S_j)=0$.

Finally, the functions $F_{j}^h$ are bounded on $U_j$ because their images are contained in any ball which contains $\Omega_U\cap M$, which is bounded. Therefore, we can extend the functions $F_j^h$ as holomorphic functions through $S_j$. Obviously, the graph of $F_j^h$ on $U_j$ coincides with the closure of its graph on $U_j\setminus S_j$; therefore, the collection of the graphs of the $F_j's$ (which are finite-dimensional analytic subspaces of $H\setminus M$ by Th\'eor\`eme 2 in the third part of \cite{R}) supports the current $[\widetilde{A}_U]$, which is then a holomorphic chain.

\medskip

The open sets $$\omega_U\ \doteqdot\ \bigcup_{k\in U} L_k$$ are a covering of $M$. Since $M$ is quasi-locally compact, we can find a locally finite countable subcovering $\omega_i$, $i\in\mathbb N$. In each $\omega_i$ lives a holomorphic chain $\tilde{A}_i$. On the intersection of two such open domains $\omega_i$ and $\omega_j$ the two chains coincide by analiticity, since they both coincide near the boundary $b\Omega$ with the manifold $W$ of Corollary~\ref{L4}.

The union $T=\tilde{A}_i$ of the holomorphic chains defined in the open sets $\omega_i$ is the holomorphic chain with boundary $M$.

\medskip

A finite dimensional analytic variety in $H$ is contained in a finite dimensional complex manifold and it is a complex variety in the latter. Therefore, we can repeat almost verbatim the argument used in \cite{DS1} to show that the singularities of $\widetilde{A}_U$ are a discrete set.

\medskip

The chain $T$ we constructed is unique in $\Omega$; this follows from the fact that there is no holomorphic cycle of $H$ of positive dimension contained in $\Omega$. Suppose $X$ is a complex analytic subset of finite dimension in $H$, which is contained in $\Omega$; if we consider the linear functional $\lambda$ defined at the beginning of this section, obviously the set $\{|\lambda|\leq\delta\}$ doesn't intersect $X$ for $\delta>0$ small enough, therefore the function $|\lambda|$ attains a positive minimum on $X$, but then the holomorphic function $\lambda$ has to be constant on $X$. Therefore $X$ lies in $\{\lambda=c\}\cap \Omega$ which is bounded. By \cite{Au}, $X$ has to be of dimension $0$, but this is impossible. {\flushright$\Box$}

\bigskip

We remark that the previous proof works also in the finite-dimensional case, giving a simplification of the argument used in \cite{DS1}, by employing the classical result by King, instead of its Hilbert space analogue.\vspace{0.3cm}

It is also worth noticing that we can to some extent relax the convexity property, asking only for $\Omega$ to be convex, strongly convex at one point and strongly pseudoconvex (or at least the Levi-form of
$b\Omega$ to have at most $m$ vanishing eigenvalues). In fact, strong convexity at one point and convexity everywhere ensure that the slices are compact, if the hyperplanes are parallel to the tangent at the point of strong convexity; we also need the strong pseudoconvexity assumption to guarantee that the Levi form of the boundary has positive eigenvalues, a fact which is implied by strong convexity but it isn't by mere convexity.

\begin{proof}[Proof of Theorem 1.2.] It is sufficient to show that if Conjecture \ref{conj87} holds true, then hypothesis (iv) is always satisfied.

We thus assume that Conjecture \ref{conj87} holds true and we consider the metric space given by $M$ with the distance $d$ given by the restriction of the distance of $H$.  Endowed with such distance, $M$ is a locally compact space; as a manifold, $M$ is second countable, therefore it is $\sigma-$compact.

As the closure of $\{\|x-x_0\|<\rho\}$ in $H$ is $\{\|x-x_0\|\leq \rho\}$ for every $x_0\in H$ and $\rho>0$, the same holds when the open and closed balls are intersected with $M$. Therefore, the metric $d$ is Heine-Borel, i.e. bounded closed sets are compact.

Now, let us take $x\in M$ and $\rho>0$; by the Heine-Borel property, the set
$$\{y\in M\ \vert\ d(x,y)\leq\rho\}=M\cap \{y\in H\ \vert\ \|x-y\|\leq\rho\}$$
is compact, that is, $M$ is quasi-locally compact, therefore we can apply Theorem \ref{main} and obtain the desired result.\end{proof}

\section{Further questions}
It would be nice to get rid of the technical hypotheses (ii), (iii), (iv) or to find examples showing that the extension does not hold without them.

Hypothesis (ii) (or its weaker version explained after the proof of Theorem 1.1) is needed to apply the cut, extend and paste method as we presented it (see example~\ref{parapertissima}), but might not be necessary to extension, as another line of proof might be possible.

Hypothesis (iii) is already present in the compact case treated in~\cite{M}, and already in that case it would be nice to see whether it is a necessary request or not.

It is worth noticing that an example showing extension does not hold just under hypotheses (i), (ii) and (iii) would be an indirect proof of the falseness of Williamson and Janos' conjecture.\vspace{0.3cm}

A possible direction for future research on the subject is that pursued in~\cite{DS2} in $\mathbb C^n$: given a (pseudo)convex domain $\Omega\subset H$, and a subdomain $A\subset b\Omega$ is it possible to find a domain $E\subset \Omega$ depending only on $\Omega$ and $A$ such that every maximally complex manifold (of real dimension at least $5$, satisfying some technical conditions) $M\subset A$ is the intersection of the boundary of a complex variety $W\subset E$ with $A$?

Thanks to what we proved in this paper, if $\Omega$ is a strongly convex domain, and $A=b\Omega$, then $E=\Omega$. Thus the question we are asking is indeed a generalization of the main result of this paper.\vspace{0.3cm}

Another quite natural question is whether it is possible to extend this result (or the one contained in \cite{M}) to Banach spaces.

\vspace{1cm}



\label{finishpage}


\begin{thebibliography}{Abc}

\bibitem{AK} L.\ Ambrosio, B.\ Kirchheim, \emph{Rectifiable sets in metric and Banach spaces}, Math.\ Ann., \textbf{318} 3 (2000), 527--555.
\bibitem{Au} V.\ Aurich, \emph{Bounded analytic sets in {B}anach spaces},Ann. Inst. Fourier, \textbf{36} 4 (1986), 229--243.
\bibitem{TDS}  G.\ Della Sala, \emph{Geometric properties of non-compact CR manifolds}, Tesi \textbf{14}, Edizioni della Normale, Pisa (2009), 103+xv.
\bibitem{DS1} G.\ Della Sala, A.\ Saracco, \emph{Non-compact boundaries of complex analytic varieties}, Int.\ J.\ Math.\ \textbf{18} 2 (2007), 203--218.
\bibitem{DS2} G.\ Della Sala, A.\ Saracco, \emph{Semi-global extension of maximally complex submanifolds}, Bull.\ Aust.\ Math.\ Soc.\ \textbf{84} (2011), 458--474.
\bibitem{D99} T.-C.\ Dinh, \emph{Conjecture de Globevnik-Stout et th\'eor\`eme de Morera pur une
            cha\^\i ne holomorphe}, Ann.\ Fac.\ Sci.\ Toulouse Math. \textbf{8} (1999) 235--257.

\bibitem{DH} P.\ Dolbeault and G. Henkin, \emph{Surfaces de Riemann de bord donn\'e dans ${\C}{\P}\sp n$}, in \textit{Contributions to complex analysis and analytic geometry}, Aspects Math.\ (Vieweg, Braunschweig, 1994), pp.\ 163--187.

\bibitem{DH2} P.\ Dolbeault and G.\ Henkin, \emph{Cha\^\i nes holomorphes de bord donn\'e dans ${\C}{\P}\sp n$}, Bull.\ Soc.\ Math.\ France \textbf{125} (1997) 383--445.
\bibitem{G} M.\ P.\ Gambaryan, \emph{Regularity condition for complex films}, Uspekhi Mat.\ Nauk \textbf{40} (1985) 203--204.
\bibitem{HL} F.\ Reese Harvey and H.\ Blaine Lawson Jr., \emph{On boundaries of complex analytic varieties. I}, Ann.\ of Math.\ (2) {\bf 102} (1975), 223--290.
\bibitem{HL3} F.\ R.\ Harvey and  H.\ B.\ Lawson, Jr., \emph{On boundaries of complex analytic varieties. II}, Ann.\ of Math.\ \textbf{106} (1977) 213--238.

\bibitem{HL2}F.\ R.\ Harvey and  H.\ B.\ Lawson, Jr., \emph{Addendum to Theorem 10.4 in \lq\lq Boundaries of analytic varieties\rq\rq}, [arXiv: math.CV/0002195] (2000).
\bibitem{Le} H.\ Lewy, \emph{On the local character of the solutions of an atypical linear
            differential equation in three variables and a related theorem
            for regular functions of two complex variables}, Ann.\ of Math.\ \textbf{64} (1956) 514--522.
\bibitem{M1} S.\ Mongodi, \emph{Some applications of metric currents to complex analysis}, Man.\ Math.\ \textbf{141} (2013) 363--390.
\bibitem{M} S.\ Mongodi, \emph{Positive metric currents and holomorphic chains in Hilbert spaces}, [arXiv:1207.5244] (2012).
\bibitem{R} G.\ Ruget, \emph{A propos des cycles analytiques de dimension infinie}, Inv. Math. {\bf 8} (1969) 267--312.
\bibitem{S} A.\ Saracco, \emph{Extension problems in complex and CR-geometry}, Tesi \textbf{9}, Edizioni della Normale, Pisa (2008), 153+xiv.
\bibitem{St} G.\ Stolzenberg, \emph{Uniform approximation on smooth curves}, Acta Math.\ \textbf{115} (1966) 185--198.

\bibitem{W} J.\ Wermer, \emph{The hull of a curve in $\C\sp{n}$}, Ann. of Math.\ \textbf{68} (1958) 550--561.

\bibitem{WJ} R.\ Williamson, L.\ Janos, \emph{Constructing metrics with the Heine-Borel property}, Proc.\ A.M.S.\ {\bf 100} 3 (1987), 567--573.




\end{thebibliography}
\end{document}